\documentclass[12pt]{article}
\usepackage{amsfonts}
\usepackage{amsmath, amssymb, eepic, graphics, color}

\newtheorem{theorem}{\bf Theorem}[section]

\newtheorem{proposition}[theorem]{\bf Proposition}

\newtheorem{conjecture}[theorem]{\bf Conjecture}

\newcommand{\eopf}{\framebox[3mm]}

\newenvironment{proof}
{\noindent{\bf Proof.}\ }%
{ \hfill\eopf\par\bigskip}%

\begin{document}
\baselineskip16pt

\centerline{{\Large { On a problem concerning integer distance graphs}}}

\bigskip

\bigskip

\centerline{Janka Oravcov\'a}
\smallskip
\centerline{{\it Department of Applied Mathematics}}
\centerline{{\it Faculty of Economics, Technical University}}
\centerline{{\it B. N\v emcovej~32, 040 01 Ko\v{s}ice, Slovak Republic}}
\centerline{\scriptsize{e-mail:{\it oravcova.j@gmail.com}}}

\begin{center}
and
\end{center}

\centerline{Roman Sot\'ak\footnote{This work was supported by the Slovak Science and Technology Assistance Agency under the contract no. APVV-0007-07.}}
\smallskip
\centerline{\it Institute of Mathematics,}
\centerline{\it Faculty of Science, P.J. \v Saf\'arik University}
\centerline{\it Jesenn\'a 5, 041 54 Ko\v sice, Slovak Republic}
\centerline{\scriptsize{e-mail:\textit{roman.sotak@upjs.sk}}}

\begin{abstract}
For $D$ being a subset of positive integers, the integer distance graph is the graph $G(D)$, whose vertex set is the set of integers, and edge set is the set of all pairs $uv$ with $|u-v| \in D$. It is known that $\chi(G(D)) \leq |D|+1$.

This article studies the problem (which is motivated by a conjecture of Zhu): "Is it true that $\chi(G(D)) = |D|+1$ implies $\omega(G(D)) \geq |D|+1$, where $\omega(H)$ is the clique number of $H$?".
We give a negative answer to this question, by showing an infinite class of integer distance graphs with $\chi(G(D))=|D|+1$ but $\omega(G(D))=|D|-1$.
%

\smallskip

\noindent{\bf Keywords:} integer distance graph, chromatic number, clique number.

\smallskip

\noindent{\bf 1991 Mathematics Subject Classification:} 05C15, 05C63, 05C75.

\smallskip
\end{abstract}

\section{Introduction}
Let $D$ be a subset of the set of positive integers $\mathbb N$. The {\em integer distance graph} $G({\mathbb Z},D) = G(D)$ is defined as the graph with vertex set $V(G(D)) = {\mathbb Z}$ the set of integers, and the edge set containing all pairs $uv$ whose absolute difference $|u-v|$ falls in the set $D$. We call $D$ the {\em distance set}.

A \textit{coloring} $f:\ V(G) \to \{f_1,f_2,\dots ,f_k\}$ of $G$ is an assignment of colors to the vertices of $G$ such that $f(u) \ne f(v)$ for all adjacent vertices $u$ and $v$. The minimum number of colors required to color the vertices of $G$ is the {\em chromatic number} $\chi(G)$ of $G$. If such a minimum does not exist we write $\chi(G) = \infty$.

For a distance set $D = \{d_1,d_2,\dots, d_k \} \subseteq \mathbb N$, we write $G(D) = G(d_1,d_2,\dots, d_k)$ and $\chi(G(D)) = \chi(D) = \chi(d_1,d_2,\dots, d_k)$.

Integer distance graphs were introduced by Eggleton, Erd\H os and Skilton \cite{EES}. They introduced it as a variation of the well--known plane coloring problem: what is the least number of colors required to color the points of Euclidean plane so that points of unit distance are colored differently?

In \cite{H}, Hadwiger constructed a tiling of the plane in seven sets of congruent hexagons such that no set contains two points of distance 1. On the other hand, there exist 4-chromatic unit distance graphs in the plane (see \cite{MM}). Therefore, we have $4 \leq \chi(G({\mathbb R}^2,\{1\})) \leq 7$. Until now, no substantial progress on this problem has been made. The chromatic number of special distance graphs has been determined in several papers (see for example. \cite{KM3,LZ}).

If $d$ is an arbitrary divisor of the elements $d_1,d_2,\dots,d_k$ of the distance set $D$, then the integer distance graph $G(D) = G(d_1,d_2,\dots,d_k)$ is isomorphic to $d$ disjoint copies of $G(\frac{d_1}{d},\frac{d_2}{d},\dots,\frac{d_k}{d})$. Hence we will restrict ourselves throughout this paper to integer distance graphs $G(D)$ such that the greatest common divisor ($gcd$ in the sequel) of the distance elements is 1.

General bounds for the chromatic number (if $D$ is nonempty) are $2 \leq \chi(D) \leq |D|+1$. The lower bound is attained for $gcd(D)=1$ if and only if all elements of $D$ are odd, proof for the upper bound can be found in \cite{W}. For example, upper bound is attained if $D=\{1,2,\dots ,n\}$, since $K_{n+1}$ is a subgraph of $G(D)$.

If $|D|=3,\ D=\{x,y,z\}$ and $gcd(D)=1$, then $\chi(D) = 4$ if and only if $D=\{1,2,3n\}$ or $D=\{x,y,x+y\}$ with $x \not \equiv y\ (\mbox{mod }3)$ (see e.g. \cite{Vo,Z1}). If $x,y,z$ are odd then $\chi(D)=2$. For all other 3-element distance sets $D$, $\chi(D) = 3$ (see \cite{Z1}).

For $|D| \geq 4$, the complete characterization of distance graphs with respect to chromatic number is not known. It seems to be an interesting question to determine all integer distance graphs whose chromatic number attains the maximum value $|D|+1$.

If $|D|=4$, then $\chi(D) = 5$ if $D=\{1,2,3,4n\}$ or $D=\{x,y,x+y,|y-x|\}$ with $x \equiv y \equiv 1\ (\mbox{mod }2)$ \cite{KM1}. It is not known whether there are other 4-element distance sets $D$ such that $\chi(D) = 5$.

It is conjectured by Zhu that if $G(D)$, $|D|\geq3$, is triangle-free, then $\chi(D)\leq|D|$ (see Conjecture 4.1 in \cite{Z1}). Liu and Zhu cited this conjecture, but their formulation is stronger (see Conjecture 5.3 in \cite{LZ2}). In this article, we study the following question which is related to these conjectures.

\bigskip

\noindent \textbf{Question A}
\textit{Whenever $\chi(D)$ attains the maximum value, then clique number $\omega(G(D))$ is at least the cardinality of $D$, in other words, $\chi(D) = |D|+1$ implies $\omega(G(D)) \geq |D|$?}

\bigskip

Kemnitz and Marangio \cite{KM2} proved that $\omega(G(D)) \geq |D|$ if and only if $|D| \leq 1$, $D = \{x,y,x+y\}$, $D=\{x,y,x+y,|y-x|\}$ or $D=\{x,2x,\dots,nx,y\}$, where $x \neq y$. If the answer to Question A is positive for all sets $D$, then the result of Kemnitz and Marangio would imply that there does not exist any other sets $D$ with $\chi(D) = |D| + 1$.

We show that the answer to Question A is not always positive by presenting infinite class of integer distance graphs $G(D)$ with $\omega(G(D)) < |D|$ and $\chi(D) = |D|+1$.

\section{Main results}

\begin{theorem}
For $D=\{1,4,5,6,7\},\ \chi(D) = 6,\ \omega(G(D)) = 4.$
\end{theorem}
\begin{proof}
The graph $G(D)$ contains $K_4$ as the subgraphs induced on the vertices $\{0,1,5,6\}$, therefore $\omega(G(D)) \geq 4$. On the other hand, by result of Kemnitz and Marangio \cite{KM2} $\omega(G(D)) < |D| = 5$ because $|D| \neq 1,\ D \neq \{x,y,x+y\},\ D \neq \{x,y,x+y,|y-x|\}$ and $D \neq \{x,2x,\dots ,nx, y\}$. Thus $\omega(G(D)) = 4$.

In the next, we will show that there is no coloring of $G(D)$ with 5 colors.
Assume that there exists a coloring $f:\ V(G(D)) \to \{a,b,c,d,e\}$. Consider the longest sequence of consecutive integers such that the colors assigned to them are all different. Since 5 colors are available for coloring, this sequence has length at most 5. On the~other hand, it has length at least 3, otherwise $f$ would be a constant coloring, or two colors would alternate, leading to the same color of two vertices in distance 4.

First, suppose that the sequence above has length 5. Without loss of generality, let $f(0) = a,\ f(1) = b,\ f(2) = c,\ f(3) = d,\ f(4) = e$. Then $f(7) = e$ (vertices 0,1,2,3 are neighbours of the vertex 7), and, similarly, $f(6) = d,\ f(5) = c,\ f(8) = a,\ f(9)=b,\ f(10)=a, f(11)=b$. But then it is not possible to color the vertex 12, because five of its neighbours are colored with all five colors.

Next suppose that the sequence above has length 4. Again, we can assume $f(0) = a,\ f(1) = b,\ f(2) = c,\ f(3) = d$; moreover, $f(1) \neq e,\ f(-1) \neq e$. Then $f(4) \in \{b,c\},\ f(-1) \in \{b,c\}$ and $f(4) \neq f(-1)$, hence, $f(4) \cup f(1) = \{b,c\}$. We have then $f(-4) = f(7) = e$, because $-4$ and $7$ are neighbours of the vertices $0,1,2,3$.  Moreover, $f(6) = d,\ f(-3) = a,\ f(5) = e$. But then no color can be assigned to the vertex $-2$, since all available colors were used to color its neighbours $-3,-1,2,3,4,5$.

Finally, suppose that this sequence has length 3. Let $f(0) = a,\ f(1)=b,\ f(2) = c$ and $f(3) \not \in \{d,e\},\ f(1) \not \in \{d,e\}$. We consider the following possibilities:
\begin{description}
\item[1)] $f(-1) = b,\ f(3) = a,\ f(4) \not \in \{d,e\}$ (otherwise there is multicolored sequence of length 4), hence, $f(4) = c$,
\item[2)] $f(-1) = c,\ f(3) = b,\ f(-2) \not \in \{d,e\}$, hence, $f(-2) = a$,
\item[3)] $f(-1) = c,\ f(3) = a,\ f(4) \not \in \{d,e\},\ f(4) = b,\ f(-2) \not \in \{d,e\}$. But then the vertex $-2$ cannot be colored.
\end{description}

The possibilities 1) and 2) are equivalent, because they generate the sequence $x y x z y z$, where $x,y,z$ are three different colors; thus, we can suppose that 1) holds. Then $f(3) \in \{d,e\}$ and (since none of those two colors was used yet) we can put $f(5) = d$. But then $f(6) = e$ gives the multicolored sequence of length 4, a contradiction.
\end{proof}

Next we will show that the distance set of this theorem is not the only one that gives a negative answer to Question A.

\begin{theorem}
For $D = \{1,2,3, \dots, 2k-1,2k+1,4k\},\ \chi(D) = |D|+1$ and $\omega(G(D)) = |D|-1$.
\end{theorem}
\begin{proof}
Since $|D| = 2k+1$ and $G(D)$ contains the induced subgraph $K_{2k}$ with vertex set $\{0,1,\dots, 2k-1 \}$, $\omega(G(D)) \geq 2k = |D| -1$. But $D$ is none of the types given by result of Kemnitz and Marangio \cite{KM2}, so $\omega(G(D)) < |D|$. Therefore, $\omega(G(D)) = |D| - 1$.

In the next, assume that there exists a coloring $f:\ V(G(D)) \to \{f_1,\dots , f_{2k+1}\}$. By \textit{sequence of length} $t$ we shall mean $t$ consecutive integers.

\begin{proposition}\label{p3}
Each sequence of length $2k$ is assigned with $2k$ distinct colors.
\end{proposition}
This is true because each such sequence induces $K_{2k}$.
\begin{proposition}\label{p4}
There exists a sequence of length $2k+1$ assigned with $2k+1$ distinct colors.
\end{proposition}
In the opposite case, the coloring $f$ is periodic with period $2k$ and $f(0) = f(2k) = f(4k)$, a contradiction with the fact that $\{0,4k\} \in E(G(D))$.

\begin{proposition}\label{p5}
There exists a sequence of length $2k+1$ assigned with $2k$ distinct colors.
\end{proposition}
This can be seen from the following: for the sequence of length $2k+1$ being assigned with all $2k+1$ colors (without loss of generality, let it be the sequence $1,2,\dots,2k+1$ and $f(i) = f_i$ for $i=1,2,\dots, 2k+1$) we obtain $f(0) = f_{2k}$ (the vertex 0 is adjacent with all the vertices $1,2,\dots,2k-1,2k+1$), hence, the sequence $0,1,\dots, 2k$ is assigned with $2k$ colors.

\begin{proposition}\label{p6}
Each sequence of length $2k+2$ is assigned with all $2k+1$ colors.
\end{proposition}
Proof is by contradiction. Let $f(1) = f_1,\ \dots, f(2k-1) = f_{2k-1},\ f(2k) = f_{2k}$. Since only $2k$ colors can be used for coloring vertices $2k+1$ and $2k+2$, it follows $f(2k+1) = f_1,\ f(2k+2)=f_2$. Consider the color of the vertex $4k+2$. With regard to its neighbourhood, we have $f(4k+2) \neq f(2),\ f(4k+2) \neq f(2k+1)$ and $f(4k+2) \not \in \{f(2k+3,f(2k+4),\dots,f(4k+1)\}$. This set is the set of colors assigned to $2k-1$ consecutive vertices, which have to be assigned with $2k-1$ distinct colors. Moreover, vertices $2k+3,\dots, 4k$ are adjacent to the vertex $2k+1$ (which is assigned with $f_1$) and the vertex $4k+1$ is adjacent with the vertex 1 (which is also assigned with $f_1$); this means that their colors are different from $f_1$. Next, vertices $2k+3,\dots ,4k+1$ are adjacent to the vertex $2k+2$ which is of the color $f_2$, so, their colors are different from $f_2$. Then $f_1,\ f_2$ and the next $2k-1$ colors cannot be assigned to the vertex $4k+2$, a contradiction.

According to the Proposition \ref{p3}, let $f(1) = f_1,\ dots, f(2k) = 2k$. Then, by Proposition \ref{p5}, we can assume $f(2k+1) = f_1$ and, by Proposition \ref{p6}, $f(2k+2) = f_{2k+1}$. We will show that $f(2k+2i+1) = f(2i+1)$ and $f(2k+2i+2) = f(2i)$, for $i=1,2,\dots, k-1$.

We proceed by induction on $i$. For $i=1$, it follows that $f(2k+3)$ is necessarily assigned with $f_3 = f(3)$ and, by Proposition \ref{p6} (for the sequence $3,4,\dots,2k+4$) $f(2k+4)$ is necessarily assigned with $f_2 = f(2)$. Hence suppose that the neighbours $2i,\ 2i+2,\dots ,2k,\ 2k+1,\ 2k+2,\dots,\ 2k+2i$ of the vertex $2k+2i+1$ are colored. For coloring of these neighbours, the colors $f_{2i},\ f_{2i+2},\dots,\ f_{2k}$, $f_1$, $f_{2k+1}$, $f_3$, $f_2$, $f_5$,  $f_4,\dots,f_{2i-1},\ f_{2i-1}$ were used, that is, all except $f_{2i+1} = f(2i+1)$. Hence, $f(2k+2i+1) = f(2i+1)$. For the vertex $2k+2i+2$, the neighbours $2i+1,\ 2i+3,\dots ,2k,\ 2k+1,\ 2k+2,\dots,2k+2i+1$ were colored with colors $f_{2i+1},\ f_{2i+3},\dots, f_{2k},\ f_1,\ f_{2k+1},\ f_3,\ f_2,\dots ,f_{2i-1},\ f_{2i-2},\ f_{2i+1}$, hence, by Proposition \ref{p6}, $f_{2k+2i+2}=f_{2i}$.

Finally, consider the color of the vertex $4k+1$. Due to its neighbourhood, $f(4k+1)\not\in\{f(1),f(2k),f(2k+2),\dots, f(4k)\} =$ $\{f_1,f_{2k},f_{2k+1}$, $f_3$, $f_2$, $f_5,f_4,\dots$, $f_{2k-1}$, $f_{2k-2}\}$. But then no color is available for the vertex $4k+1$. This finishes the proof.
\end{proof}

We believe that the answer to Question A is positive in all cases except of the ones described in Theorems 1 and 2. For supporting this, all distance sets $D$ with $4 \leq |D| \leq 11$ and $\max D \leq b_{|D|}$ were checked to fulfil it (we used $b_4 = 2000,\ b_5 = 800,\ b_6 = 300,\ b_7 = 100$, $b_8=120$, $b_9=90$, $b_{10}=75$, $b_{11}=40$).

\begin{conjecture}\label{nasa}
Let $gcd(D) =1$ and $\chi(D) = |D|+1$. Then $D=\{1,4,5,6,7\}$, $D=\{1,2,3,\dots, 2k-1,2k+1,4k\}$ or $\omega(G(D)) \geq |D|$.
\end{conjecture}

\noindent In support of this conjecture we indicate that for the sets of shape $\{1,2,\dots,k,t\cdot a\}\backslash\{a\}$, where $a\leq k <t\cdot a$, there is no other case with $\chi(D) = |D| + 1$ and $\omega(G(D)) < |D|$ (for more details see \cite{OS}).



\begin{thebibliography}{99}

\bibitem{EES}
R.B. Eggleton, P. Erd\H os and D.K. Skilton, {\em Coloring the real line}, J.~Comb. Theory, Series B \textbf{39} (1985), 86--100.\

\bibitem{H}
H. Hadwiger, {\em Ungel\H oste Probleme No. 40}, Elemente der Math. \textbf{16} (1961), 103--104.\

\bibitem{KM1}
A.~Kemnitz and M.~Marangio, {\em Chromatic numbers of integer instance graphs}, Discrete Math. \textbf{233} (2001), 239--246.\

\bibitem{KM2}
A.~Kemnitz and M.~Marangio, {\em Colorings and list colorings of integer distance graphs}, Congr. Numer. \textbf{151} (2001), 75--84.\

\bibitem{KM3}
A.~Kemnitz and M.~Marangio, {\em Edge colorings and total colorings of integer distance graphs}, Discussiones Mathematicae Graph Theory \textbf{22} (2002), 149--158.\

\bibitem{LZ}
D. D.-F. Liu and X. Zhu, {\em Distance graphs with missing multiples in the distance sets}, J.~Graph Theory \textbf{30} (1999), 245--259.\

\bibitem{LZ2}
D. D.-F. Liu and X. Zhu, {\em Fractional Chromatic Number and Circular Chromatic Number for Distance Graphs with Large Clique Size}, J.~Graph Theory \textbf{47} (2004), 129--146.\

\bibitem{MM}
L. Moser and W. Moser, {\em Solution to problem 10}, Canad. Math. Bull. \textbf{4} (1961), 187--189.\

\bibitem{OS}
J. Oravcov\'a and Sot\'ak, {\em Integer distance graphs and Lonely runner - Tight sets and tight vectors}, preprint.\

\bibitem{Vo}
M. Voigt, {\em Colouring of distance graphs}, Ars Combinatoria \textbf{52} (1999), 3--12.\

\bibitem{W}
H.~Walther, {\em {\H U}ber eine spezielle Klasse unendlicher Graphen}, In: K. Wagner und R. Bodendiek. Graphentheorie, Bd.2. Bibl. Inst., Mannheim 1990, pp. 268--295.\

\bibitem{Z1}
X. Zhu, {\em Circular chromatic number of distance graphs with distance sets of cardinality $3$}, J.~Graph Theory \textbf{41} (2002), 195--207.\
\end{thebibliography}
\end{document}